\documentclass[12pt,twoside,a4paper,reqno]{amsart}

\usepackage{a4,enumerate}
\usepackage[noadjust]{cite}
\usepackage{amssymb,amsmath,amsthm}
\usepackage{comment}
\usepackage{calc}
\usepackage{xcolor}
\usepackage{tikz,tikz-cd}

\newcommand\lhead{H. Vogt, J. Voigt}
\newcommand\rhead{Sequences of sectorial forms}
\swapnumbers
\numberwithin{equation}{section}

\newtheorem{theorem}{Theorem}[section]

\newtheorem{proposition}[theorem]{Proposition}

\theoremstyle{definition}

\newtheorem{remarks}[theorem]{Remarks}
\newtheorem{example}[theorem]{Example}

 \mathchardef\ordinarycolon\mathcode`\:
  \mathcode`\:=\string"8000
  \begingroup \catcode`\:=\active
    \gdef:{\mathrel{\mathop\ordinarycolon}}
  \endgroup

\newcommand\rlim{
\mathchoice{\vcenter{\hbox{${\scriptstyle{+}}$}}}
{\vcenter{\hbox{$\scriptstyle{+}$}}}
{\vcenter{\hbox{$\scriptscriptstyle{+}$}}}
{\vcenter{\hbox{$\scriptscriptstyle{+}$}}}}

\newcommand\smid{\nonscript \mskip2mu plus2mu {\mid}%
\nonscript \mskip2mu plus2mu}     
\newcommand\scpr[2]{{(#1\smid#2)}}
\def\bigscpr(#1,#2){{\left(#1\nonscript \mskip2mu plus2mu \middle| \nonscript 
\mskip2mu
plus2mu#2\right)}}

\newcommand*\id{\operatorname{id}}

\newcommand\ran{R}
\renewcommand\ran{\operatorname{\rm ran}}

\newcommand*\ex{\exists\+}
\newcommand*\scdot{\mkern2mu{\cdot}\mkern2mu}
\newcommand*\textstar[1]{\raise0.1ex\hbox{$\scriptstyle#1\mkern-.5mu$}}
\newcommand*\tstar{{\textstar{*}}}

\renewcommand{\Re}{\operatorname{Re}}

\newcommand\imu{{\rm i}}
\renewcommand\phi{\varphi}

\renewcommand\epsilon{\varepsilon}

\newcommand*\dist{\operatorname{dist}}

\newcommand{\R}{\mathbb{R}\nonscript\hskip.03em}
\newcommand{\N}{\mathbb{N}\nonscript\hskip.03em}
\newcommand{\C}{\mathbb{C}\nonscript\hskip.03em}

\newcommand\cA{\mathcal A}

\newcommand\cL{\mathcal L}

\newcommand\la{\lambda}

\newcommand*\tj{\tilde{\jmath}\+}
\newcommand*\tk{\tilde k}
\newcommand*\ta{\tilde a}

\newcommand\tmo{^{-1}}
\newcommand\toh{^{1/2}}      

\newcommand\norm[1]{\|#1\|}

 \newcommand\ca{\check a}
 \newcommand\cu{\check u}

\DeclareFontFamily{U}{mathx}{\hyphenchar\font45}
\DeclareFontShape{U}{mathx}{m}{n}{
      <5> <6> <7> <8> <9> <10>
      <10.95> <12> <14.4> <17.28> <20.74> <24.88>
      mathx10
      }{}
\DeclareSymbolFont{mathx}{U}{mathx}{m}{n}
\DeclareFontSubstitution{U}{mathx}{m}{n}
\DeclareMathAccent{\widecheck}{0}{mathx}{"71}

\DeclareMathAccent{\widecheck}{0}{mathx}{"71}

\newcommand*\wcVhelper{\rlap{$\widecheck{\rule{0pt}{1.5ex}\rule{0.83em}{0pt}}$}V}
\newcommand*\wcV{\mathchoice{\wcVhelper}{\wcVhelper}{\widecheck V}{\widecheck V}}

\newcommand*\embed{\hookrightarrow}

\newcommand*\ccdotp{{\cdot}\mkern1mu}

\newcommand\dupa[2]{\left\langle #1, #2 \right\rangle}

\newcommand\eul{{\rm e}}

\def\formE(#1,#2){\sum_{e\in E}\int_{a_e}^{b_e} #1_e'(x)\ol{#2_e'(x)}\,dx}

\makeatletter
\let\qedhere@ams\qedhere
\def\qedhere{\@ifnextchar[{\@qedhere}{\qedhere@ams}}
\def\@qedhere[#1]{\tag*{\raisebox{-#1ex}{\qedhere@ams}}}
\makeatletter
\def\env@cases{%
  \let\@ifnextchar\new@ifnextchar
  \left\lbrace
  \def\arraystretch{1.1}%
  \array{@{\,}l@{\quad}l@{}}%
}
\renewcommand\section{\@startsection {section}{1}{\z@}%
                                     {-3.25ex \@plus -1ex \@minus -.2ex}%
                                     {1.5ex \@plus.2ex}%
                                     {\normalfont\large\bfseries}}
\makeatother

\newcommand\restrict{\vphantom f\mskip1mu\vrule\mskip2mu}
\newcommand\set[2]{\bigl\{#1{;}\penalty300\;#2\bigr\}}
\newcommand\bset[2]{\Bigl\{#1{;}\;#2\Bigr\}}

\newcommand\ol{\overline}

\newcommand\dom{\mathop{\rm dom}}

\newcommand\+{\mkern1.5mu}
\newcommand\<{\mkern-1.5mu}

\newcommand\rfrac[2]{\tfrac{#1}{\raisebox{0.1em}{$\scriptstyle#2$}}}

\newcommand\comp{{\textnormal c}}
\newcommand\Cci{{\displaystyle 
C_{\raise0.2ex\hbox{$\scriptstyle\comp$}}^\infty}}

\renewcommand\le{\leqslant}

\renewcommand\ge{\geqslant}

\newcommand{\from}{{:}\;}
\renewcommand{\from}{\colon}

\newcommand\sse{\subseteq}

\newcommand\di{\mathclose{}\,\mathrm{d}}

\newcommand\slim{\mathop{\rm s\kern.08em\mbox{\rm -}lim}} 

\newcommand\abstracttext{\noindent
We present a form convergence theorem for sequences of sectorial forms and their associated semigroups in a complex Hilbert space. Roughly speaking, the approximating forms $a_n$ are all `bounded below' by the limiting form $a$, but in contrast to the  previous literature there is no monotonicity hypothesis on the sequence. Moreover, the forms are not supposed to be closed or densely defined.

For a sectorial form one obtains an associated linear relation, whose negative generates a degenerate strongly continuous semigroup of linear operators. Our hypotheses on the sequence of forms imply strong resolvent convergence of the associated linear relations, which in turn implies convergence of the corresponding semigroups. The result is illustrated by two examples, one of them closely related to the Galerkin method of numerical analysis.
\vspace{8pt}

\noindent
MSC 2010: 47A07, 47B44, 47D06
\vspace{2pt}

\noindent
Keywords: sectorial form, strong resolvent convergence, degenerate strongly continuous semigroup, m-sectorial operator
}
\begin{document}

\title{On sequences of sectorial forms converging `from above'}

\author{Hendrik Vogt and
        J\"urgen Voigt}
        
\date{}

\maketitle

\begin{abstract}
\abstracttext
\end{abstract}

\medskip
\emph{Dedicated to Jerry Goldstein on the occasion of his 80th
birthday}

\section{Introduction}
\label{intro}

The history of form convergence theorems goes back to at least the 1950's. The setup we consider involves a sequence $(a_n)$ of sectorial forms in a Hilbert space where all $a_n$ are `bounded below' in a suitable sense by a sectorial form~$a$. Then one seeks conditions implying that the operators $A_n$ associated with $a_n$ converge to the operator $A$ associated with $a$. We refer to Kato's book \cite[Chap.~VIII, Theorem~3.6]{Kato1980} for a fundamental result concerning this topic as well as for an account of the previous history. 

A special case of this kind of results involves decreasing sequences
of symmetric forms; an interesting feature is that then one does not need to specify the limiting form $a$ in advance. Actually, this is a touchy issue because the limiting form obtained pointwise need not be closable; see the cautious formulation in \cite[Chap.~VIII, Theorem~3.11]{Kato1980}. This issue was resolved by Simon in \cite[Theorem~3.2]{Simon1978}, where the `regular part' of non-closable symmetric forms was introduced. 

It appears that the topic of convergence `from above' for the more general case of sequences of sectorial forms was taken up only later, by Arendt and ter Elst \cite{ArendtElst2012}. One of the new features in that paper is that forms and the association of operators are treated in a different setup, where closedness and closability of forms -- which played an important role in the previous treatments -- are no longer relevant. 

We add a comment on form convergence theorems for
\emph{increasing} sequences of forms. This topic is closely related,
but employs rather different tools; we refer to 
\cite[Chap.~VIII, Theorem~3.13]{Kato1980}, \cite[Theorems~3.1 and 4.1]{Simon1978},
\cite[Theorem~5]{Ouhabaz1995}, \cite[Theorems~1.2, 2.2 and~3.2]{BattyElst2014}, 
\cite[Theorems~4.1 and~5.1]{VogtVoigt2020}. An essential feature in these results is that the domains of the sequence $(a_n)$ of forms satisfy $\dom(a_n)\supseteq\dom(a_{n+1})$ for all $n\in\N$.
\smallskip

Our theorem stated below is a generalisation of 
\cite[Chap.~VIII, Theorem~3.6]{Kato1980} and \cite[Theorem~3.7]{ArendtElst2012}. We refer to Remark~\ref{rems-orm-conv-fa}(c) for a discussion of the differences between this theorem and the previous results.
 
\begin{theorem}\label{thm-form-conv-fa}
Let $H$ be a complex Hilbert space, and let $(a,j)$  be a quasi-sectorial form 
in $H$. For $n\in\N$ let $a_n$ be a form with $\dom(a_n)\sse\dom(a)$. Let 
$\theta\in [0,\frac\pi 2)$, and assume that
\begin{equation}\label{eq-unif-est}
a_n(u)-a(u)\in\ol{\Sigma_\theta}\qquad \bigl(u\in\dom(a_n),\ 
n\in\N\bigr).
\end{equation}
Let $D$ be a core for $a$, and suppose that for all $u\in D$
there exists a sequence $(u_n)$ in $\dom(a)$, $u_n\in\dom(a_n)$ for all $n\in\N$, such that
$u_n\to u$ in $\dom(a)$ and $a_n(u_n)-a(u_n) \to 0$ as $n\to\infty$.

Let $A$ be the linear relation associated with $(a,j)$, 
and let $A_n$ be the 
linear relation associated with $(a_n,j\restrict_{\dom(a_n)})$, for $n\in\N$. 
Then $(A_n)$ converges to $A$ in the strong resolvent sense, 
i.e.~$(\la+A_n)\tmo\to(\la+A)\tmo$ ($n\to\infty$) strongly
for all $\la>-\gamma$, where $\gamma$ is a vertex of $(a,j)$. 
\end{theorem}

Clearly, this statement of the theorem asks for quite a number of clarifications. These will be given in Section~\ref{sec-prelims}.

We point out two important issues concerning our result. Firstly, there is no kind of monotonicity required for the sequence 
$(a_n)$; see 
Remark~\ref{rems-orm-conv-fa}(c). Secondly, the result implies that the sequence of semigroups generated by the linear relations $-A_n$ converges to the semigroup generated by $-A$; see Section~\ref{sec-conv-dscsg}.

In Section~\ref{sec-prelims} we explain the concept of (not necessarily densely defined) quasi-sectorial forms and their associated linear relations.

Section~\ref{sec-proof} is devoted to the proof of Theorem~\ref{thm-form-conv-fa}, including preparations. We also include some comments on the hypotheses of the theorem as well as an additional result concerning norm convergence of the resolvents in 
Theorem~\ref{thm-form-conv-fa}.

In Section~\ref{sec-conv-dscsg} we describe the degenerate strongly continuous semigroups associated with the forms $a_n$ in the theorem and show that these semigroups converge to the semigroup associated with $a$.

Section~\ref{sec-examples} contains examples illustrating the main theorem.

\section{Sectorial forms and associated linear relations}
\label{sec-prelims}

The topic of this section is a review of sesquilinear forms and their associated linear relations; this review is a brief presentation and extension of notions and results from~\cite{ArendtElst2012}. Let $H$ be a complex Hilbert space.

Let $a$ be a sesquilinear form on some complex vector space, called $\dom(a)$. We say that $a$ is \emph{sectorial} of angle $\theta\in[0,\pi/2)$ if
\[
 a(u)\in\ol{\Sigma_\theta} \qquad(u\in\dom(a)),
\]
where $\Sigma_\theta:=\set{r\eul^{\imu\alpha}} {r>0,\ \alpha\in(-\theta,\theta)\cup\{0\}}$. (Note that our definition of `sectorial' is slightly more restrictive than
the one used in \cite[Chap.~V, \S\,3.10]{Kato1980}.)

Additionally, let $j\colon \dom(a)\to H$ be a linear operator; then the couple $(a,j)$ is called a \emph{form in} $H$. The form $(a,j)$ is called \emph{quasi-sectorial} of angle $\theta$ and with vertex $\gamma\in\R$ if
\[
 a(u) \in \gamma + \ol{\Sigma_\theta}\qquad (u\in\dom(a),\ \norm{j(u)}=1),
\]
or equivalently
\[
 a(u)-\gamma\norm{j(u)}^2\in\ol{\Sigma_\theta}\qquad (u\in\dom(a)).
\]

Let $(a,j)$ be a sectorial form of angle $\theta$ in $H$.
Then
\begin{equation}\label{eq-s-i-prod}
 \scpr xy_{a,j}:=(\Re a)(x,y)+\scpr {j(x)}{j(y)}_H\qquad (x,y\in\dom(a))
\end{equation}
defines a semi-inner product on $\dom(a)$, where $\Re a$ is defined by $(\Re a)(x,y):=\frac12(a(x,y)+\ol{a(y,x)})$.
Let $(V,q)$ be the completion of $(\dom(a),\scpr\cdot\cdot_{a,j})$, which means that $V$ is a Hilbert space and $q\colon\dom(a)\to V$ is a linear operator with dense range and $\scpr xy_{a,j}=\scpr{q(x)}{q(y)}_V$ for all $x,y\in\dom(a)$. Then the mappings $j$ and $a$ possess unique continuous `extensions' $\tj\colon V\to H$ and $\ta\colon V\times V\to\C$, satisfying $\tj(q(x))=j(x)$ and $\ta(q(x),q(y))=a(x,y)$ for all $x,y\in\dom(a)$. More generally, if $W$ is a Banach space and $L\colon (\dom(a),\norm\cdot_{a,j})\to W$ is a continuous linear mapping, then there exists a (unique) continuous linear operator $\tilde L\colon V\to W$ such that $\tilde L(q(x))=Lx$ for all $x\in\dom(a)$. (This property will be needed in the proof Theorem~\ref{thm-form-conv-fa}.)

Now suppose additionally that $(a,j)$ is \emph{densely defined}, i.e.~$\ran(j)$ is dense in~$H$. 
With the above preparations one then defines the \emph{m-sectorial operator $A$ associated with $(a,j)$} by
\begin{equation}\label{eq-ass-op-nc}
A:=\set{(x,y)\in H\times H} {\ex u\in V\colon \tj(u)=x,\ \ta(u,v)=\scpr{y}{\tj(v)}_H \ (v\in V)}.
\end{equation}
Note that $A$ is also the operator associated with the form $(\ta,\tj)$ in $H$.

To complete the picture -- still for the case of sectorial forms -- it remains to discuss the case when $\ran(j)$ is not dense in $H$. In this case the above procedure can be carried out with $H$ replaced by $H_1:=\ol{\ran(j)}$, and the formula corresponding to \eqref{eq-ass-op-nc} yields an m-sectorial operator $A_1$ in $H_1$. Then the formula \eqref{eq-ass-op-nc}, as it stands, yields the linear relation $A=A_1\oplus(\{0\}\times H_1^\perp)$ in $H$, which we call the \emph{linear relation associated with 
$(a,j)$}. The direct sum in this description is an orthogonal direct sum in $H\times H$. 

If $(a,j)$ is quasi-sectorial with vertex $\gamma$, then one obtains the linear relation $A$ associated with $(a,j)$ by first applying the previous procedure to the sectorial 
form~$(a_{-\gamma},j)$, $\dom(a_{-\gamma}):=\dom(a)$,
\[
 a_{-\gamma}(x,y):=a(x,y)-\gamma\scpr{j(x)}{j(y)}_H \qquad(x,y\in\dom(a)),
\]
with associated linear relation $A_{-\gamma}$, and then adding $\gamma I$ leads to 
\[
 A:=A_{-\gamma} + \gamma I:=\set{(x,y+\gamma x)}{(x,y)\in A_{-\gamma}}.
\]
In this case the semi-inner product \eqref{eq-s-i-prod} on $\dom(a)$ becomes
\begin{equation}\label{eq-s-i-prod-gamma}
 \scpr xy_{a_{-\gamma},j} = \Re a(x,y)+ (1-\gamma)\scpr {j(x)}{j(y)}_H\qquad (x,y\in\dom(a)).
\end{equation}
The definition of $A$ given above does not depend on the choice of the vertex $\gamma$.

As a consequence, if $(a,j)$ is quasi-sectorial and $\lambda\in\R$, then the linear relation 
associated with the form $a_\lambda$, $a_\lambda(x,y)=a(x,y)+ \lambda\scpr{j(x)}{j(y)}_H$ ($x,y\in\dom(a)$), is given by $A+\lambda I=\set{(x,y+\lambda x)}{(x,y)\in A}$.

The inverse of a linear relation $B\sse H\times H$ is the linear
relation $B\tmo:=\set{(y,x)}{(x,y)\in B}$. We point out that the resolvents 
$(\lambda+A_n)\tmo=(A_n+\lambda I)\tmo$ and $(\lambda+A)\tmo=(A+\lambda I)\tmo$ of $-A_n$ 
and~$-A$ at the point~$\lambda$, in Theorem~\ref{thm-form-conv-fa}, in fact are operators; see Remark~\ref{extended-cea-rems}(d).

\section{Proof of Theorem \ref{thm-form-conv-fa}}
\label{sec-proof}

The following result provides a key estimate needed in the proof
of Theorem~\ref{thm-form-conv-fa}.

\begin{proposition}\label{prop-cea-extended}
Let~$V$ be a complex Hilbert space, and let $a$ be a bounded coercive form on~$V$,
\[
  |{a(u,v)}|\le M\norm{u}_V\norm{v}_V, \quad \Re a(u)\ge\alpha\norm{u}_V^2
  \qquad (u,v\in V)
\]
for some $M\ge0$, $\alpha>0$.

Let $\wcV$ be a complex Hilbert space, $\ca$ a bounded coercive form on~$\wcV$, $J\in\cL(\wcV,V)$, and assume that there
exists $\theta\in[0,\frac\pi2)$ such that
\[
  \ca(v)-a(Jv) \in \ol{\Sigma_\theta} \qquad (v\in \wcV).
\]

Let $\eta\in V^*$ (the antidual space of $V$), and let $u\in V$,  $\cu\in \wcV$ be the unique elements such that
\[
  a(u,v) = \dupa\eta v_{V^*,V} \quad (v\in V), \qquad  \ca(\cu,v) = \dupa\eta{Jv}_{V^*,V} \quad (v\in \wcV).
\]

Then
\begin{equation}\label{eq-galerkin-conv-rate}
  \norm{u-J\cu}_V^2
  \le \inf_{v\in\wcV}\Bigl(\frac{M^2}{\alpha^2} \norm{u-Jv}_V^2 + \frac{c^2}{2\alpha}\bigl|\ca(v)-a(Jv)\bigr|\Bigr),
\end{equation}
where $c:=1+\tan\theta$.
\end{proposition}

\begin{proof} The existence of $u$ and $\cu$ is a consequence of the Lax--Milgram lemma.

We define a form $b$ on~$\wcV$ by
\[
  b(w,v) := \ca(w,v) - a(Jw,Jv).
\]
The assumptions imply that $b$ is sectorial of angle $\theta$; hence \cite[Chap.~VI, (1.15)]{Kato1980} implies
\begin{equation}\label{bn-est}
  |b(w,v)| \le c(\Re b(w))\toh (\Re b(v))\toh \qquad (w,v\in \wcV).
\end{equation}
For $v\in \wcV$ we compute 
\[
  a(u-J\cu,Jv) = \dupa\eta{Jv}_{V^*,V} - a(J\cu,Jv)
  = \ca(\cu,v) - a(J\cu,Jv) = b(\cu,v).
\]
Hence, for $v\in\wcV$ we obtain
\begin{align*}
\alpha\norm{u-J\cu}_V^2
 &\le \Re a\bigl(u-J\cu,(u-Jv)+J(v-\cu)\bigr) \\
 &  = \Re a(u-J\cu,u-Jv) + \Re b(\cu,v-\cu).
\end{align*}
Now we use~\eqref{bn-est}, the boundedness of~$a$ and twice Young's inequality ($ab\le\frac12\bigl(\gamma a^2+\rfrac1\gamma b^2)$ for all $a,b\ge0,\ \gamma>0$) to estimate
\begin{align*}
\alpha\norm{u-J\cu}_V^2
 &\le M\norm{u-J\cu}_V \norm{u-Jv}_V + c(\Re b(\cu))\toh(\Re b(v))\toh
      - \Re b(\cu) \\
 &\le \frac\alpha2 \norm{u-J\cu}_V^2 + \frac{M^2}{2\alpha} \norm{u-Jv}_V^2
      + \frac{c^2}4 \Re b(v).
\end{align*}
Reshuffling terms we conclude that
\begin{equation*}
  \norm{u-J\cu}_V^2
  \le \frac{M^2}{\alpha^2} \norm{u-Jv}_V^2 + \frac{c^2}{2\alpha} \Re\bigl(\ca(v)-a(Jv)\bigr). 
\qedhere
\end{equation*}
\end{proof}

\begin{remarks}\label{extended-cea-rems}
(a) If the space $\wcV$ is a subspace of~$V,\ 
J \from \wcV \embed V$ is the embedding and $\ca = a\restrict_{\wcV\times \wcV}$,
then Proposition~\ref{prop-cea-extended} reduces to C\'ea's lemma; see~\cite[p.\;365, Proposition~3.1]{Cea1964} (with a slightly different version of the resulting inequality). In this case \eqref{eq-galerkin-conv-rate} reduces to
\[
\norm{u-\cu}_V \le \frac M \alpha \inf_{v\in\wcV}\norm{u-v}_V,
\]
expressing that, up to a constant not depending on $\wcV$, the approximate solution $\cu$ is as close as possible to the solution $u$.
C\'ea's lemma is a tool for the Galerkin method of numerical analysis; see also 
Example~\ref{exple1}.

(b) 
We define the \emph{Lax--Milgram operator} associated with $a$ as the bounded linear operator 
$\cA\colon V\to V^*,\ u\mapsto a(u,\ccdotp)$.
As the form $a$ is coercive, the Lax--Milgram lemma implies that $\cA$ is an isomorphism.

(c) The following considerations serve to transform \eqref{eq-galerkin-conv-rate} into an inequality that will be used in the proof of Theorem~\ref{thm-form-conv-fa} given below.
In the setting of Proposition~\ref{prop-cea-extended}, let $\cA\in\cL(V,V^*)$ and $\check\cA\in\cL(\wcV,{\wcV}^*)$ be
the Lax--Milgram operators associated with $a$ and $\check a$, respectively.
We define the dual operator $J' \in \cL(V^*,{\wcV}^*)$ of~$J$ by
\[
  J'\eta := \eta\circ J\qquad(\eta\in V^*).
\]
Then for $\eta\in V^*$, the elements $u\in V$ and $\check u\in\wcV$ specified in Proposition~\ref{prop-cea-extended} are given by
\[
 u=\cA\tmo\eta,\quad \check u=\check\cA\tmo J'\eta.
\]
With this notation
\eqref{eq-galerkin-conv-rate} reads as
\begin{equation}\label{eq-galerkin-conv-rate-spelled-out}
   \norm{\cA\tmo\eta-J\check\cA\tmo J'\eta}_V^2
  \le \inf_{v\in\wcV}\Bigl(\frac{M^2}{\alpha^2}\norm{\cA\tmo\eta-Jv}_V^2 + \frac{c^2}{2\alpha}\bigl|\ca(v)-a(Jv)\bigr|\Bigr).
\end{equation}

(d) Let $(a,j)$ be a quasi-sectorial form with vertex $\gamma>0$ in a complex Hilbert space $H$, let $V$, $(\ta,\tj)$ and the associated linear relation $A$ be as described in Section~\ref{sec-prelims}, and let $\cA$ be the Lax--Milgram operator associated with the form $\ta$. Define $\tk\in \cL(H,V^*)$ by $\tk(y):=\scpr y{\tj(\cdot)}_H$ for $y\in H$. Then $A\tmo=\tj \cA\tmo \tk$. Note that this formula shows that $A\tmo$ is an operator.

Indeed, by \eqref{eq-ass-op-nc} one has $(x,y)\in A$ if and only if there exists 
$u\in V$ such that $\tj(u)=x$ and $\ta(u,v) = \scpr y{\tj(v)}$ for all $v\in V$. The 
latter property is equivalent to $\cA u = \tk(y)$, i.e.~$x=\tj(u)=\tj\cA\tmo \tk(y)$. (See also \cite[Proposition~2.1]{VogtVoigt2018} concerning this interplay between $A$ and the Lax--Milgram operator $\cA$.) 
\end{remarks}

\begin{proof}[Proof of Theorem~\ref{thm-form-conv-fa}]
Note that the hypotheses of the theorem imply that the forms $a_n$ are quasi-sectorial with the same vertex~$\gamma$ as~$a$ and with a common angle.

(i) In this main step of the proof we suppose that $\gamma>0$, and we show the assertion for 
$\lambda=0$ ($>-\gamma$), i.e.~we show that $A_n\tmo\to A\tmo$ ($n\to\infty$) strongly. 
We will use the representation of $A\tmo$ and of $A_n\tmo$
described in Remark~\ref{extended-cea-rems}(d).
The hypothesis $\gamma>0$ implies that $\Re a$ is a semi-inner product on~$\dom(a)$
that is equivalent to the semi-inner product $\scpr\scdot\scdot_{a,j}$ defined in~\eqref{eq-s-i-prod}.
Moreover $(a,j)$ is sectorial, and \eqref{eq-unif-est} implies that
$(a_n,j)$ is sectorial for all  $n\in\N$.

\setbox0=\vtop{%
$\begin{tikzcd}[row sep=0pt, column sep=0.63em]
\smash{\dom(a)} \arrow[r, "q"] & \<V\!\! \arrow[dr, "\tj" pos=0.4] & \\
&& \!H \\
\smash{\dom(a_n)}\<\vphantom V \arrow[swap, r, "q_n\vphantom t"] \arrow[hook, 
uu, "\id"] &
\!V_n\!\! \arrow[swap, ur, "\tj_n" pos=0.3] \arrow[uu, "J_n\strut"] &
\end{tikzcd}$}%
\parshape 6
0pt \textwidth
0pt \textwidth
0pt \textwidth
0pt \textwidth
0pt \textwidth
0pt 0.7\textwidth
Let $(V,q)$ denote the completion of $(\dom(a),\Re a)$.
Then there exist a unique $\tj\in\cL(V,H)$ and a unique bounded form
$\tilde a\colon V\times V\to\C$ such that $\tj\circ q=j$ and
$\tilde a(q(u),q(v))=a(u,v)$ for all $u,v\in \dom(a)$.
Analogously we define $V_n$, $q_n$, $\tj_n$ and $\tilde a_n$
corresponding to $a_n$ and $j_n:=j\restrict_{\dom(a_n)}$, for $n\in\N$.
It follows from~\eqref{eq-unif-est} that the embedding
$\dom(a_n)\hookrightarrow\dom(a)$ is continuous for all $n\in\N$,
and from the description of the completion in Section~\ref{sec-prelims} it follows that
\vadjust{\moveright0.73\textwidth \vbox to -\dp0{\vskip-0.5\baselineskip\box0}}%
there exists $J_n\in\cL(V_n,V)$ such
that $J_n\circ q_n=q\restrict_{\dom(a_n)}$. Then $\tj_n\circ
q_n=j\restrict_{\dom(a_n)}=\tj\circ q\restrict_{\dom(a_n)}=\tj\circ J_n\circ
q_n$; hence $\tj_n=\tj\circ J_n$ on $\ran(q_n)$, and by denseness on all of
$V_n$.

\parshape 0
Let $n\in\N$. Then for all $u\in\dom(a_n)$ we have
\[
  \tilde a_n(q_n(u))-\tilde a(J_nq_n(u)) = a_n(u)-a(u) \in \ol{\Sigma_\theta}.
\]
Since $\ran(q_n)$ is dense in $V_n$, it follows that
$\tilde a_n(v)-\tilde a(J_nv) \in \ol{\Sigma_\theta}$ for all $v\in V_n$.

For the application of Proposition~\ref{prop-cea-extended} we note that the convergence hypothesis of the theorem is equivalent to requiring that
\begin{equation}\label{eq-conv-prop}
 \inf_{v\in\dom(a_n)}\bigl(\norm{u-v}_a^2 + \bigl|a_n(v)-a(v)\bigr|\bigr)\to 0\qquad (n\to\infty),
\end{equation}
for all 
$u\in D$. As $D$ is a core for $a$, the convergence \eqref{eq-conv-prop} carries over to all $u\in\dom(a)$.
In view of the properties of the mappings $q$, $q_n$ and $J_n$ the convergence \eqref{eq-conv-prop} can be rewritten as
\[
 \inf_{v\in \ran(q_n)}\bigl(\norm{u-J_nv}_V^2 + \bigl|\tilde a_n(v) - \tilde a(J_nv)\bigr|\bigr)\to 0\qquad (n\to\infty),
\]
for all $u\in\ran(q)$. Then, using the inclusions $\ran(q_n)\sse V_n$ as well as the
denseness of $\ran(q)$ in $V$, one also obtains
\begin{equation}\label{eq-a-approximation}
 \inf_{v\in V_n}\bigl(\norm{u-J_nv}_V^2 + \bigl|\tilde a_n(v) - \tilde a(J_nv)\bigr|\bigr)\to 0\qquad (n\to\infty),
\end{equation}
for all $u\in V$.

Let $\tilde\cA \in\cL(V,V^*)$ and $\tilde\cA_n\in\cL(V_n,V_n^\tstar)$ ($n\in\N$) be
the Lax--Milgram operators associated with $\tilde a$ and $\tilde a_n$ ($n\in\N$),
respectively.
By Remark~\ref{extended-cea-rems}(d) we have $A\tmo = \tj\tilde\cA\tmo\tilde k$ and
$A_n\tmo = \tj_n\tilde\cA_n\tmo\tilde k_n$ for all $n\in\N$,
where $\tilde ky = \scpr{y}{\tj(\cdot)}_H$ and $\tilde k_ny = \scpr{y}{\tj_n(\cdot)}_H$
for all $y\in H$.
Combining \eqref{eq-a-approximation} with inequality \eqref{eq-galerkin-conv-rate-spelled-out} in Remark~\ref{extended-cea-rems}(c) we conclude that
$J_n\cA_n\tmo J'_n\to\cA\tmo$ strongly in $\cL(V^*,V)$, as $n\to\infty$.

Note that $\tj_n = \tj\circ J_n$ implies $\tilde k_n = J_n'\circ \tilde k$, for all $n\in\N$.
Therefore
\begin{equation}\label{eq-sr-conv}
  A_n\tmo = \tj(J_n\cA_n\tmo J_n')\tilde k \to \tj\cA\tmo\tilde k = A\tmo\qquad(n\to\infty)
\end{equation}
strongly in $\cL(H)$. 

(ii) For the general case let $\lambda>-\gamma$. As in Section~\ref{sec-prelims} we define the form $a_\lambda$ by
$\dom(a_\lambda):=\dom(a)$,
\[
 a_\lambda(x,y):=a(x,y)+\lambda\scpr{j(x)}{j(y)}_H\qquad(x,y\in\dom(a));
\]
then $a_\lambda$ has the vertex $\gamma+\lambda>0$. Defining $a_{n,\lambda}$ correspondingly for $n\in\N$, we apply step (i) to the form $a_\lambda$ and the sequence $(a_{n,\lambda})_{n\in\N}$. This yields $A_{n,\lambda}\tmo\to A_\lambda\tmo$ ($n\to\infty$) strongly, where 
$A_{n,\lambda}$, 
$A_\lambda$ are associated with the forms $a_{n,\lambda}$, $a_\lambda$. As $A_{n,\lambda}=A_n+\lambda I$ ($n\in\N$) and $A_\lambda=A+\lambda I$ by Section~\ref{sec-prelims}, we obtain the assertion of the theorem.
\end{proof}

\begin{remarks}\label{rems-orm-conv-fa}
(a) In the hypotheses of Theorem~\ref{thm-form-conv-fa}, the condition $u_n\to u$ in $\dom(a)$ implies that $a(u_n)\to a(u)$. Therefore `$a_n(u_n)-a(u_n) \to 0$ as $n\to\infty$' is equivalent to requiring `$a_n(u_n)\to a(u)$ as $n\to\infty$'.

(b) In the proof of Theorem~\ref{thm-form-conv-fa} it is shown that the convergence property required for all $u\in D$ is in fact satisfied for all $u\in\dom(a)$; see the
argument connected with the reformulation \eqref{eq-conv-prop} of this property.

(c) The hypotheses of the theorem express a kind of convergence $a_n\to a$.
In previous results on form convergence `from above' (see, e.g., \cite[Chap.~VIII, Theorem~3.6]{Kato1980} and \cite[Theorem~3.7]{ArendtElst2012}) a more restrictive hypothesis is used, namely that the set
\[
 D:=\bset{u\in \bigcup_{k\in\N}\bigcap_{n\ge k}\dom(a_n)} {a_n(u)\to a(u)\ (n\to\infty)}
\]
is a core for~$a$. In particular, the elements of $D$ belong to almost every 
$\dom(a_n)$. This kind of `monotonicity' for the domains of the forms is not required in Theorem~\ref{thm-form-conv-fa}.

Another important weakening in our hypotheses, with respect to the quoted sources, is that we do not require the forms $a_n$ to be densely defined. The inclusion of this feature in our result was motivated by the well-established Galerkin method; see Example~\ref{exple1}. The use of non-densely defined (symmetric) forms in the context of form convergence theorems was already promoted in~\cite{Simon1978}.
\end{remarks}

The following result is a generalisation of \cite[Theorem~3.8]{ArendtElst2012};
it is a simple by-product of our method of proof for 
Theorem~\ref{thm-form-conv-fa}.

\begin{theorem}\label{thm-form-conv-fa-j-comp}
Let the hypotheses be given as in Theorem~\ref{thm-form-conv-fa}. Assume additionally that the operator $j\colon (\dom(a),\scpr \cdot\cdot_{a_{-\gamma},j})\to H$ is compact, where 
$\scpr \cdot\cdot_{a_{-\gamma},j}$ is the semi-inner product 
\eqref{eq-s-i-prod-gamma}, with a vertex $\gamma$ of $a$.

Then the conclusion of strong resolvent convergence of the linear relations $A_n$ to $A$ can be strengthened to norm resolvent convergence, i.e.~$(\la+A_n)\tmo\to(\la+A)\tmo$ ($n\to\infty$) in 
$\cL(H)$ for all $\la>-\gamma$.
\end{theorem}

\begin{proof}
As in the proof of Theorem~\ref{thm-form-conv-fa} it is sufficient to consider the case 
when $\gamma>0$ is a vertex of $a$ and $\lambda=0$, and we adopt the notation used in that proof.

The properties of the completion of $\dom(a)$ imply that the image under~$q$ of the open unit ball of $\dom(a)$ is dense in the open unit ball of~$V$. Therefore the compactness of~$j$ implies that~$\tj$ is a compact operator as well. Now we go to the last two paragraphs of step (i) in the proof of Theorem~\ref{thm-form-conv-fa} and observe that, by Schauder's theorem on the adjoint of compact operators, the operator~$\tk$ is compact. Therefore the strong convergence $J_n\cA_n\tmo J'_n\to\cA\tmo$ together with \eqref{eq-sr-conv} implies that 
$A_n\tmo\to A\tmo$ in~$\cL(H)$.
\end{proof}

\section{Degenerate strongly continuous semigroups}
\label{sec-conv-dscsg}

Let $X$ be a Banach space, and let $T\colon[0,\infty)\to\cL(X)$ be a one-parameter semigroup on~$X$; note that then
$P:=T(0)$ is a projection. We 
call $T$ a \emph{degenerate strongly continuous semigroup} if $T(0)=\slim_{t\to0\rlim}T(t)$.
A degenerate strongly continuous semigroup $T$ is the direct sum of a $C_0$-semigroup $T_1$ on $X_1:=P(X)$ and the semigroup on $X_0:=(I-P)(X)$ that is identically zero; in fact~$T_1$ is the restriction of~$T$ to~$X_1$. Let~$A_1$ be the generator of the $C_0$-semigroup~$T_1$. Then it is meaningful to call the linear relation 
$A:= A_1\oplus(\{0_{X_0}\}\times X_0)=\set{(x,A_1x+y)}{x\in\dom(A_1),\ y\in X_0}$ 
the \emph{generator} of the degenerate strongly continuous semigroup~$T$.
We refer to \cite{Arendt2001} for the introduction and investigation of degenerate semigroups. Arendt does not define the generator for degenerate strongly continuous semigroups, but rather deals with the concept of pseudo-resolvents, which are the resolvents of linear relations~$A$ as above.

Let $H$ be a complex Hilbert space, and let $(a,j)$ be a quasi-sectorial form in $H$. We define $H_1:=\ol{\ran(j)}$ and $H_0:=H_1^\perp$. In Section~\ref{sec-prelims} it was described that then the linear relation $A$ associated with $(a,j)$ is of the form $A_1\oplus(\{0\}\times H_0)$, where $A_1$ is a quasi-m-sectorial operator, and as such the operator $-A_1$ is the generator of a $C_0$-semigroup on $H_1$. Hence, in the terms described above, the linear relation $-A$ is the generator of a degenerate strongly semigroup $T$, the direct sum of the $C_0$-semigroup $T_1$ on $H_1$ generated by $-A_1$ and the zero-semigroup on~$H_0$. 

\begin{theorem}\label{thm-norm-res-con}
Let the hypotheses be given as in Theorem~\ref{thm-form-conv-fa}. Denote by~$T$ the degenerate strongly continuous semigroup generated by the linear relation~$-A$, and for $n\in\N$ denote by~$T_n$ the degenerate strongly continuous semigroup generated by~$-A_n$.
Then 
\begin{equation}\label{eq-conv-sg}
 T_n(t)x \to T(t)x\qquad(n\to\infty)
\end{equation}
uniformly on compact subsets of \,$[0,\infty)$, for all $x\in H$. 
\end{theorem}

\begin{proof}
We refer to \cite[Theorem~4.2]{Arendt2001} for the result that the strong resolvent convergence stated in Theorem~\ref{thm-form-conv-fa} implies the convergence 
\eqref{eq-conv-sg} for all $x\in\ran(T(0))=\ol{\ran(j)}$. However, on $\ran(j)^\perp$, the semigroup $T$ and the semigroups $T_n$ act as the zero operator (note that $\ran(j)^\perp\sse\ran(j\restrict_{\dom(a_n)})^\perp$); hence the convergence \eqref{eq-conv-sg} on 
$\ran(j)^\perp$ is trivial.
\end{proof}

\section{Examples}
\label{sec-examples}

\begin{example}\label{exple1}
Let $H$ be a complex Hilbert space, and let $(a,j)$
be an embedded sectorial form in $H$. Here, `embedded' means that $\dom(a)\sse H$ and that 
$j\colon\dom(a)\embed H$ is the embedding. Then
\[
 \scpr xy_{a,j}:=a(x,y)+\scpr xy_H\qquad(x,y\in \dom(a))
\]
defines a scalar product on $V:=\dom(a)$.

Let $(V_n)_{n\in\N}$ be a sequence of subspaces of $V$, and define $a_n$ as the restriction of $a$ to $V_n$. Suppose that for all $x\in V$ one has 
$\dist_{\norm\cdot_{a,j}}(x,V_n)\to0$ as $n\to\infty$. Then Theorem~\ref{thm-form-conv-fa} implies that the linear relations $A_n$ associated with $a_n$ converge to the linear relation $A$ associated with $a$, in the strong resolvent sense.

We add three comments: 1.~It is not difficult to construct cases in which the spaces $V_n$ are such that $V_n\cap V_m=\{0\}$ for all $n\ne m$; in particular, the condition mentioned in Remark~\ref{rems-orm-conv-fa}(c) is not satisfied.
2.~For the proof of  this application of Theorem~\ref{thm-form-conv-fa} one would need Proposition~\ref{prop-cea-extended} only in the simplified form of C\'ea's lemma alluded to in Remark~\ref{extended-cea-rems}(a).
3.~If all the spaces $V_n$ are finite-dimensional, then the assertion in the example corresponds to the Galerkin method of numerical analysis.
\end{example}

\begin{example}\label{exple2}
 Let $\Omega\sse\R^n$ be an open set, and let $(\Omega_k)_{k\in\N}$ be a sequence of open subsets of $\Omega$ with the property
that for each compact set $K\sse\Omega$ there exists $k_K\in\N$ such that $K\sse\Omega_k$ for all $k\ge k_K$.

Let $a$ be the classical Dirichlet form on $\Cci(\Omega)$,
\[
 a(f,g)=\int\nabla f\cdot\ol{\nabla g}\di x\qquad (f,g\in\Cci(\Omega)),
\]
let $a_k$ be the restriction of $a$ to $\dom(a_k):=\Cci(\Omega_k)$, and let $A$ and $A_k$ ($k\in\N$) be the self-adjoint linear relation associated with
$a$ and $a_k$ ($k\in\N$), respectively. This means that the operator $A$ is the negative Dirichlet Laplacian in $L_2(\Omega)$, whereas for $k\in\N$, $A_k$ is the negative 
Dirichlet Laplacian in $L_2(\Omega_k)$ supplemented by $\{0\}\times L_2(\Omega\setminus\Omega_k)$ to a self-adjoint linear relation in $L_2(\Omega)$. (We refer to \cite[Section~5]{Arens1961} for the first appearance and the analysis of self-adjoint linear relations in Hilbert spaces.)

Then Theorem~\ref{thm-form-conv-fa} (or Example~\ref{exple1}) implies that $A_k\to A$ in the strong resolvent sense. If additionally $\Omega$ is bounded, then Theorem~\ref{thm-form-conv-fa-j-comp} in tandem with the Rellich--Kondrachov theorem implies that $A_k\to A$ in the norm resolvent sense.

Note that, in general, in the above context one cannot expect to 
find a decreasing subsequence of $(a_k)_k$. However, the condition mentioned in Remark~\ref{rems-orm-conv-fa}(c) is satisfied. We refer to \cite[Section~6]{Arendt2001} for a treatment of semigroup convergence in a similar context, 
where the approximating sets $\Omega_k$ need not be subsets of $\Omega$, but $\Omega$ is required to satisfy a regularity property.
It is not difficult to see that in the proof given in \cite{Arendt2001} this regularity property can be removed if all the $\Omega_k$ are subsets of~$\Omega$. 
\end{example}

\bigskip

\noindent
Hendrik Vogt\\
Fachbereich Mathematik, Universit\"at Bremen\\ 
Postfach 330 440\\
28359 Bremen, Germany\\
{\tt 
hendrik.vo\rlap{\textcolor{white}{hugo@egon}}gt@uni-\rlap{\textcolor{white}{%
hannover}}bremen.de}\\[3ex]
J\"urgen Voigt\\
Technische Universit\"at Dresden, Fakult\"at Mathematik\\
01062 Dresden, Germany\\
{\tt 
juer\rlap{\textcolor{white}{xxxxx}}gen.vo\rlap{\textcolor{white}{yyyyyyyyyy}}%
igt@tu-dr\rlap{\textcolor{white}{%
zzzzzzzzz}}esden.de}

\end{document}